\documentclass[10pt]{amsart}
\usepackage[cp1251]{inputenc}
\usepackage[english,russian]{babel}
\usepackage{amsmath}
\usepackage{amssymb}
\usepackage{amsfonts}
\usepackage[left=3cm,right=3cm,
    top=3cm,bottom=3cm,bindingoffset=0cm]{geometry}

\newtheorem{lemma}{Lemma}
\newtheorem{theorem}{Theorem}

\newtheorem{corollary}{Corollary}
\newtheorem{proposition}{Proposition}

\newcommand{\FF}{\mathbb{F}}

\begin{document}
\renewcommand{\refname}{References}
\renewcommand{\proofname}{Proof.}
\renewcommand{\figurename}{Fig.}

\thispagestyle{empty}

\title[$L_{\infty}$ norm minimization for eigenvectors of the block graphs of STSs]{$L_{\infty}$ norm minimization for nowhere-zero integer eigenvectors of the block graphs of Steiner triple systems and Johnson graphs}
\author{{E.A. Bespalov, I.Yu. Mogilnykh, K.V. Vorob'ev}}%
\address{Evgeny Andreevich Bespalov
\newline\hphantom{iii} Sobolev Institute of Mathematics,
\newline\hphantom{iii} pr. Koptyuga, 4,
\newline\hphantom{iii} 630090, Novosibirsk, Russia}%
\email{bespalovpes@mail.ru}%
\address{Ivan Yurevich Mogilnykh
\newline\hphantom{iii} Sobolev Institute of Mathematics,
\newline\hphantom{iii} pr. Koptyuga, 4,
\newline\hphantom{iii} 630090, Novosibirsk, Russia}%
\email{ivmog84@gmail.com}%

\address{Konstantin Vasil'evich Vorob'ev
\newline\hphantom{iii} Sobolev Institute of Mathematics,
\newline\hphantom{iii} pr. Koptyuga, 4,
\newline\hphantom{iii} 630090, Novosibirsk, Russia}%
\email{konstantin.vorobev@gmail.com}%

\thanks{\sc Bespalov, E.A., Mogilnykh, I.Yu., Vorob'ev K.V.
}
\thanks{\copyright \ 2023 Bespalov, E.A., Mogilnykh, I.Yu., Vorob'ev K.V.}
\thanks{\rm This work was funded by the Russian Science
Foundation under grant 22-21-00135, https://rscf.ru/project/22-21-00135/}

 \vspace{1cm}
\maketitle {\small
\begin{quote}
\noindent{\sc Abstract. }  We study nowhere-zero integer eigenvectors of the block graphs of Steiner triple systems and the Johnson graphs. For the first eigenvalue we obtain the minimums
of the $L_{\infty}$ norm for several infinite series of Johnson graphs, including $J(n,3)$ for all $n\geq 63$, as well as general upper and lower bounds. 
 The minimization of the $L_{\infty}$ norm for nowhere-zero integer eigenvec\-tors  with the second eigenvalue of the block graph of a Steiner triple system $S$ is equivalent to finding the minimum nowhere-zero flow for Steiner triple system $S$. For the all Assmuss-Mattson Steiner 
triple systems of the orders greater or equal to $99$ we prove that the minimum flow is bounded above by $5$. 

\medskip

\noindent{\bf Keywords:}  Steiner triple system, flow, strongly regular graph, Johnson graph, Grassmann graph, eigenvalue
 \end{quote}
}

\section{Introduction}

A vector is called {\it nowhere-zero integer} (shortly NZI vector) if all of its elements are nonzero integers.
The infinity norm $\| v \|_{\infty}$ of a vector $v$ is defined  as the maximum of the absolute values of its elements.

Let $W_S$ be the point-block incidence matrix of a Steiner triple system  $S$. A nowhere-zero integer vector $u$ such that $W_Su=0$ is called a nowhere-zero  $(\|u\|_{\infty}+1)$-{\it flow} for the Steiner triple system $S$ \cite{Akbari}. 
 It is not hard to see that the right nullspace of the incidence matrix $W_S$ coincides with the second eigenspace of the block graph of $S$ (see Proposition \ref{2eigchar} below).  

The minimum of the $L_{\infty}$ norm for flows of Steiner triple systems were considered in works \cite{Akbari}, \cite{Akbari2013}, \cite{AMMW}. Akbari, Burgess, Danziger and  Mendelsohn \cite{Akbari2017} showed that the minimum of the norm of the flows in any Steiner triple system of order $n$ is upper bounded by $O(n^2)$.

 On the other hand, studies show that for particular families of Steiner triple systems, the actual minimum of the norm of the flows is much smaller than $O(n^2)$ and this fact finds similarities in a conjecture of Tutte on existence of a $5$-flow for  the graphs \cite{Tutte}.  Speaking more precisely, all Steiner triple systems of order $15$ have $3$-flows \cite{Akbari}. Furthermore, it was proven that some well-known recursive classes of Steiner triple systems, such as direct product construction,  $2v+7$-construction admit a $3$-flow, given a $3$-flow in the original Steiner triple system \cite{AMMW}.
 As for the $2v+1$-construction, the resulting Steiner triple systems has $3$, $4$ or $5$-flow \cite{Akbari2017}. In Section 5 we establish that the Assmuss-Mattson \cite{AM} Steiner triple system obtained from any Steiner triple system $S$ of order at least $49$ has a $5$-flow, regardless of the flow in the original  system $S$. 

One might consider a more general definition of a flow for any given natural matrix $W$, which is, for example, the inclusion matrix of subsets \cite{Akbari} or subspaces \cite{SS}. A flow in  these cases is the sum of eigenvectors of a Johnson or Grassmann graph with specific eigenvalues. For example, let $W$ be the inclusion matrix of $2$-subsets and $k$-subsets of $n$-element set. In \cite{Akbari} it was shown that 
for $k=3$, there is a  nowhere-zero integer vector $v$ such that $Wv=0$, $\|v\|_{\infty}=2$, i.e. a generalized $3$-flow. Note that for $k=3$ any nonzero vector $v$ such that $Wv=0$ is an eigenvector of the Johnson graph $J(n,3)$ with the third eigenvalue. Relying on the properties of higher order inclusion matrices, the authors of \cite{Akbari} extended the result and showed that a $3$-flow exists for the inclusion matrix of $2$-subsets and $k$-subsets for any $k\geq 3$.

The perspective of the continuing studies of flows for Steiner triple systems and structural matrices implies the following natural question.

{\bf Problem 1.} Given a distance-regular graph $\Gamma$ of diameter $d$ and its rational eigenvalue $\theta_i(\Gamma),$ $0\leq i\leq d$,
find $$min\{\|u\|_{\infty}{\it+1}:u \mbox{ is a nowhere-zero integer } \theta_i(\Gamma)\mbox{-eigenvector of }\Gamma\},$$ which we denote as $m(i,\Gamma)$ in below.

From results of Akbari et al. there is always a solution for Problem 1.

\begin{theorem}  \cite{Akbari}, \cite{Akbari2013} 
Let $\Gamma$ be a distance-regular graph with a rational eigenvalue $\theta_i(\Gamma)$, $0\leq i\leq d$. Then there is a NZI $\theta_i(\Gamma)$-eigenvector of $\Gamma$.
\end{theorem}
\begin{proof}
By  \cite[Theorem 3]{Akbari2013} there is a nowhere-zero real eigenvector for every eigenvalue of any distance-regular graph.
By \cite[Lemma 3.3]{Akbari} the existence of a nowhere-zero real vector, belonging to the null space of a rational matrix implies the existence of nowhere-zero  integer vector in the null space.
By taking the matrix to be $A_{\Gamma}-\theta_iI$, where $A_{\Gamma}$ is the adjacency matrix of $\Gamma$, we obtain the required.
\end{proof}

Definitions, notation and basic theory are in Section 2. The results of Section 3 are presented in a general context: for the $q$-ary Steiner triple systems and Grassmann graphs; the classic Steiner triple systems and Johnson graph are treated as a particular case. Recently, the $q$-ary Steiner triple systems were shown to exist asymptoti\-cally \cite{Keevash}, however there is only one explicit example \cite{EOW} of order $13$. We consider a  description of the eigenspaces of the block graph of any $q$-ary Steiner triple system in terms of the point-block incidence matrix of $STS$. In particular, for the first eigenvalues, we see that the eigenvectors of the Grassmann graph $J_q(n,3)$ are in a natural one-to-one correspondence with that of the block graph of the $q$-ary Steiner triple system $S$ order $n$: the restriction of any eigenvector of $J_q(n,3)$ to the blocks of $S$ is an eigenvector of the block graph of $S$. This relation between the first eigenspaces of STSs and $J_q(n,k)$ is in spirit of \cite{Vor}, where an extension of the eigenvectors of Johnson graphs to that of Hamming graphs was established. Despite this strong connection, in Section 4 we show that the minimums of the $L_{\infty}$ norm for nowhere-zero eigenvectors for both graphs are different for $q=1$. We establish lower and upper bounds on the optimum norms of the NZI $\theta_1(J(n,k))$-eigenvectors  of the Johnson graphs $J(n,k)$ and obtain the exact minimums for infinite series of Johnson graphs $J(n,k)$. In particular, we completely solve the problem for $k=3$ and all $n>63$ (see Theorem \ref{p:jn3}). Bounds on the $L_{\infty}$ norm of the NZI eigenvectors of the block graphs of STSs with the second eigenvalue (which is equivalent to finding $i$-flow for STSs for small $i$) are given in Section 5. We start Section 5.1 with reviewing the existing results for flows in the projective and the Bose Steiner triple systems which utilize the aspects of cyclicity and resolvability of these designs. In Section 5.2 we show that  any Steiner triple system constructed by Assmuss-Mattson approach \cite{AM} of order at least $99$ has a $5$-flow. The results of this section are described in terms of flows of Steiner triple systems rather than eigenvectors of their block graphs.

 In Section 6 we discuss  completely regular codes in the block graphs of Steiner triple systems. These objects are in the scope of current study, as the minimum possible value of the $L_{\infty}$ norm  of nowhere-zero integer eigenvectors is attained on a vector arising from a specific completely regular code (see Proposition \ref{CRC1}). From the perspective of Cameron-Liebler line classes \cite{CL}, the completely regular codes in the block graphs of the Steiner triple systems with the covering radius $1$ and the first eigenvalue are of interest as they provide one of different variations \cite{DIMS}, \cite{BDIM} of such objects. 
The block graph of the projective (Hamming) Steiner triple system of order $2^r-1$ is isomorphic to the Grassmann graph $J_2(r,2)$ and all such completely regular codes are exactly Cameron-Liebler line classes in $PG(n-1,2)$. We conjecture that these codes comprise only the following classic examples of Cameron-Liebler line classes: a point, hyperplane and nonincident point-hyperplane pair (see Problem 2 in Section 6)  and show that there are no other codes for the STSs of orders $13$ and $15$. These codes in the block graphs of Steiner triple systems and affine Steiner triple systems in particular, 
were considered in \cite{GP}, where these objects and the conjecture above were treated from the perspective of small support eigenvectors of the block graphs.
% As for the completely regular codes in the block graphs of a Steiner triple system $S$ with %the covering radius $1$ and second eigenvalue, we show that such codes are equivalent to a %$1$-%subdesign of Steiner triple system $S$.

\section{Definitions and notations}
\subsection{Johnson, Grassmann graphs, $q$-ary Steiner triple systems and their block graphs}
A regular graph of diameter $d$ is called {\it distance-regular} if there is an array of integers 
$$\{\beta_0,\ldots,\beta_{d-1};\gamma_1,\ldots,\gamma_{d}\},$$
 such that for any vertices $x$ and $y$ at distance $i$, $i\in\{0,\ldots,d\}$ there are
exactly $\beta_i$ neighbors of $y$ at distance $i+1$ from $x$ and $\gamma_i$  neighbors of $y$ at distance $i-1$  from $x$. The array of integers $\{\beta_0,\ldots,\beta_{d-1};\gamma_1,\ldots,\gamma_{d}\}$
is called the {\it intersection array} of the distance-regular graph $\Gamma$.
We say that a nonzero vector $v$ is a $\theta(\Gamma)$-{\it eigenvector} if $A_{\Gamma}v=\theta(\Gamma)v$, where $A_{\Gamma}$ is the adjacency matrix of $\Gamma$. In this case $\theta$ is called an {\it eigenvalue} of $\Gamma$.
It is well known that any distance-regular graph of diameter $d$ has exactly $d+1$ distinct eigenvalues, which we index
 in descending order: $\theta_0(\Gamma)>\theta_1(\Gamma)>\ldots>\theta_d(\Gamma)$. Note that $\theta_0(\Gamma)$ is the valency of the graph $\Gamma$.

The vertices of {\it  the Grassmann graph} $J_q(n,k)$ are  $k$-subspaces of the finite vector space $\FF_q^n$ over the field $\FF_q$ and the edges are the pairs of subspaces meeting in $(k-1)$-subspace. 
We also include a limit case of $q=1$ as we define {\it the Johnson graph} $J(n,k)$ (also denoted by $J_1(n,k)$) to be the graph with the vertex set being $k$-subsets of $\{1,\ldots,n\}$ and edges being pairs of subsets meeting in $(k-1)$ set. 
Let $[^n_k]_q$ be the Gaussian binomial coefficient and $[^n_k]_1$ be the ordinary binomial coefficient.
The $k+1$ eigenvalues of $J_q(n,k)$, $q \geq 1$  are as follows:
	$$ \theta_i(J_q(n,k))=q^{k+1}[^{k-i}_1]_q[^{n-k-i}_1]_q-[^i_1]_q, 0\leq i\leq k.$$

By a $q$-ary {\it Steiner triple system} (briefly, STS) $S$ we mean a collection of $3$-subspaces (called blocks) of $\FF_q^n$ such that any $2$-subspace of $\FF_q^n$ is in exactly one subspace in $S$. 
By letting $q=1$ we include Steiner triple system in the traditional sense in the definition above, i.e. a collection of $3$-subsets of $\{1,\ldots,n\}$ (called blocks or triples) such that any $2$-subset is in exactly one block.
A $1$-subspace of $\FF_q^n$ ($1$-subset of $\{1,\ldots,n\}$) is called {\it a point} of $S$.
By {\it the order} of $S$, we mean the dimension $n$ of the ambient space $F_q^n$ (the number of points for $q=1$).
A Steiner triple system of order $n$ (in the classical sense)  is called {\it resolvable} if its blocks are parted into {\it parallel classes}, i.e. collections of pairwise nonintersecting blocks with the size of each class equal to $\frac{n}{3}$.

{\it The incidence matrix}  of a $q$-ary Steiner triple system $S$ is the matrix $W_S$, whose rows are indexed by the points of $S$ and the columns are indexed by its blocks is defined as follows:
\begin{equation*}
(W_S)_{i,T}=
 \begin{cases}
   1, &i\subseteq T\\   
   0, &\text{otherwise}
 \end{cases}.
\end{equation*}

Similarly, we define the point-block incidence matrix $W$ of $J_q(n,k)$ with the rows indexed by the points ($1$-subspaces) and the columns are indexed by the $k$-subspaces of $\FF_q^n$.
The {\it block intersection} graph of $q$-ary, $q\geq 1$, STS $S$, denoted by $\Gamma_S$,
has the blocks of $S$ as vertices and distinct blocks having a nonempty intersection as the edges. It is well-known that this graph is strongly-regular \cite{AlievSeiden} and has the following eigenvalues:
$$\theta_1(\Gamma_S)=\frac{[^{n}_1]_q-1}{[^3_1]_q-1}-[^3_1]_q-1,$$
$$\theta_2(\Gamma_S)=-[^3_1]_q.$$

The blocks of any $q$-ary STS, $q\geq 1$, of order $n$ could be treated as a set of vertices of $J_q(n,3)$. So the block intersection graph can be viewed as the subgraph of the distance-$2$ graph of $J_q(n,3)$  induced by the blocks of STS. 

%\subsection{Nowhere-zero integer eigenvectors}

\subsection{Completely regular codes}

Given $C\subseteq V(\Gamma)$, the {\it distance partition} with respect to  $C$ is $C_0=C,\ldots,C_{\rho}$ such that
$$C_i=\{x:d(x,C)=i\}.$$

The maximum of all $i$'s is denoted by $\rho$ and is called the {\it covering radius} of $C$.
A subset $C\subseteq V(\Gamma)$  is called a {\it completely regular code}  if there are numbers $\alpha_0, \ldots,\alpha_\rho$, $\beta_0,\ldots,\beta_{\rho-1}$, $\gamma_1,\ldots,\gamma_{\rho}$ such that any vertex of $C_i$
is adjacent to exactly $\alpha_i$, $\beta_i$, $\gamma_i$ vertices of $C_{i}, C_{i+1}$ and $C_{i-1}$ respectively, for $i=0, \dots \rho$ and $\gamma_0=\beta_{\rho+1}=0.$ 
The array $\{\beta_0,\ldots,\beta_{\rho-1};\gamma_1,\ldots, \gamma_{\rho}\}$  is called  {\it the intersection array} of the completely
regular code $C$.
The following tridiagonal $(\rho + 1) \times (\rho + 1)$
matrix
$$\left(%
\begin{array}{cccccc}
\alpha_0 & \beta_0& 0&0 &\ldots &0\\
\gamma_1&\alpha_1 & \beta_1&0 &\ldots &0\\
 .& .&.&.&.&.\\
.& .&.&.&.&.\\
0&.&0& \gamma_{\rho-1}&\alpha_{\rho-1} & \beta_{\rho-1}\\

0&.&0&0& \gamma_{\rho}&\alpha_{\rho}\end{array}%
\right),$$
 is called the {\it intersection matrix}
of the completely regular code $C$.  

The eigenvalues of this matrix are {\it the eigenvalues} of the completely regular code $C$. It is well-known that the eigenvalues of any completely regular code in a regular graph are necessarily eigenvalues of the graph, which is known as Lloyd's theorem. 

Given a set of vertices $C$ in a graph $\Gamma$ we denote by $\chi_C$ its characteristic vector in the vertex set of the graph.

\begin{proposition}\label{2values}(Folklore)
Let $C$ be a completely regular code with $\rho=1$,  eigenvalue $\theta_i(\Gamma)$, $i\neq 0$ and intersection array $\{\beta_0;\gamma_1\}$ in a distance-regular graph $\Gamma$. Then the vector $\beta_0\chi_C-\gamma_1\chi_{V(\Gamma)\setminus C}$ is a $\theta_i(\Gamma)$-eigenvector of $\Gamma$. Moreover, any $\theta_i(\Gamma)$-eigenvector (up to  multiplicity) taking only two values can be obtained in this manner.
\end{proposition}
%\begin{proof}
%Sketch. Let $C$ be a completely regular code with $\rho=1$,  eigenvalue $\theta_i(\Gamma)$ intersection array $\{\beta_0,\gamma_1\}$. Then the vector %$\beta_0\chi_C-\gamma_1\chi_{V(\Gamma)\setminus C}$ is a $\theta_i$-eigenvector

%\end{proof}
The current study of the completely regular codes from the point of view of Problem 1 is inspired by the following statement.

\begin{proposition}\label{CRC1}
Let $\Gamma$ be a distance-regular graph.  Then we have $m(i,\Gamma)\geq 2$. Moreover the equality $m(i,\Gamma)=2$ holds if and only if  there is 
a completely regular code with $\rho=1$, intersection array $\{\beta_0;\beta_0\}$ and eigenvalue $\theta_i(\Gamma)$ in $\Gamma$.
\end{proposition}
\begin{proof}
We obviously have $m(i,\Gamma)\geq 2$ and $m(i,\Gamma)$ attains the lower bound $2$ if and only if there is a $\theta_i(\Gamma)$-eigenvector taking values $+1$ and $-1$ only. 
From Proposition \ref{2values} we obtain the required. 
\end{proof}

For the completely regular codes with $\rho=1$ and the second eigenvalue, we see that they are equivalent to  1-subdesigns of the considered Steiner triple system.

\begin{proposition}\label{CRCeg2}
Let $S$ be a Steiner triple system of order $n$. A set $S'\subset S$ is a completely regular code with $\rho=1$ and eigenvalue $\theta_2(\Gamma_S)$ if and only if $S'$ is a 1-design.
\end{proposition}
\begin{proof}
Let $S'$ be a subset of blocks of $S$. The set $S'$ is $1$-$(n,3,\lambda)$-design if and only if the vector $W_S((\frac{n-1}{2}-\lambda)\chi_{S'}-\lambda\chi_{S\setminus S'})$ is all-zero, where $W_S$ is  the point-block incidence matrix of STS $S$. This follows from the fact that any point of $S$ belongs to $\frac{n-1}{2}$ and $\lambda$ blocks of $S$ and $S'$. In view of Proposition \ref{2eigchar}.1 below this is equivalent to the vector $(\frac{n-1}{2}-\lambda)\chi_{S'}-\lambda\chi_{S\setminus S'}$ being a $\theta_2(\Gamma_S)$-eigenvector. As any two-valued eigenvector of the graph corresponds to a completely regular code with $\rho=1$, see Proposition \ref{2values}, we obtain the required.

\end{proof}

\section{A description of the eigenspaces of the block graphs of STS and the first eigenspace of $J_q(n,k)$}

% Given a $q$-ary STS $S$ and its point-block incidence matrix $W_S$, a nowhere-zero integer valued vector $v$ such that $W_Sv=0$ is called {\it zero-sum} $(\| v \|_{\infty}+1)$ {\it flow} of $S$ \cite{Akbari}. %We call a nowhere-zero integer-valued $\theta_i(\Gamma)$-eigenvector $v$ of $\Gamma$ a 
 %$(i,\| v \|_{\infty}+1)$ {\it generalized flow} (or $(i,\| v \|_{\infty}+1)$-{\it eigenflow}) of  $\Gamma$.

Firstly, we consider the following  auxiliary fact.
 
\begin{theorem}\label{TAM}\cite[Theorem 1]{Amog} Let $\overline{\Gamma}$ be a biregular bipartite graph with valencies $c$ and $c'$, the halved
graphs $\Gamma$ and $\Gamma'$ and let $I$ denote the $|V(\Gamma)|\times |V(\Gamma')|$ incidence matrix of two parts of $\overline{\Gamma}$.
Let any pair of vertices of $V(\Gamma)$ at distance $2$ in $\overline{\Gamma}$ have exactly $m$ common neighbors and any pair of vertices of $V(\Gamma')$ at distance $2$ in $\overline{\Gamma}$ have exactly  $m'$  common neighbors. The following holds

1. Let $u$ be a $\theta$-eigenvector of $\Gamma'$,  $\theta\neq-\frac{c'}{m'}$. Then the vector $Iu$
 is a $\frac{c'-c+m'\theta}{m}$-eigenvector of $\Gamma$.

2. Given a nonzero vector $u$, $Iu$ is the all-zero vector if and only if $u$ is $-\frac{c'}{m'}$-eigenvector of $\Gamma'$.
\end{theorem}

{\it Example 1}. Consider the biregular graph whose parts are the points and the blocks of a $q$-ary STS $S$ with adjacency being point-block inclusion.
The halved graphs are the complete graph $K_{[^n_1]_q}$  on the points of $S$ (as $\Gamma$ in Theorem \ref{TAM}) and  the block graph $\Gamma_S$ of $S$ (as $\Gamma'$ in Theorem \ref{TAM}), 
$I$ is the point-block incidence matrix $W_S$ of the design $S$.
It is not hard to see that parameters in Theorem \ref{TAM}  are $m=m'=1$, $c=\frac{[^n_1]_q-1}{[^3_1]_q-1}$, $c'=[^3_1]_q$. 

We obtain the following relation between the null-spaces of the incidence matrix of the Steiner triple systems and the second eigenspace of their block graphs.  

\begin{proposition}\label{2eigchar}
Let $S$ be a $q$-ary STS $S$, $q\geq 1$ with the point-block intersection matrix $W_S$. Then

1.  A nonzero vector u fulfills $W_S u$=0 if and only if  $u$ is a
$\theta_2(\Gamma_S)$-eigenvector of $\Gamma_S$.

2. A vector $v$ is a $(\| v \|_{\infty}+1)$-flow for $S$ if and only if $v$ is a
nowhere-zero integer $\theta_2(\Gamma_S)$-eigenvector of the block graph $\Gamma_S$ of the Steiner triple system $S$. 
\end{proposition}
\begin{proof}
We apply  Theorem \ref{TAM}.2 for the graph in Example 1.
\end{proof}

\subsection{The first eigenspaces of the block graphs of STSs, Grassmann and Johnson graphs}
 
%Denote by $U$ the set of all $(-1)$-eigenvectors of the graph $J_q(n,1)$, which coincides with the complete graph $K_{[^n_1]_q}$. It is not hard to see that
%$U$ consists of all real vectors, indexed by the vertices of $K_{[^n_1]_q}$ with the sum of the values being zero.
In the theorem below by $W$ we denote the incidence matrix of all $1$-subspaces  of $\FF_q^n$  ($1$-subsets of $\{1,\ldots,n\}$ for $q=1$)
vs $k$-subspaces $\FF_q^n$  ($1$-subsets of $\{1,\ldots,n\}$ for $q=1$). 
A vector is a  {\it restriction} of a vector $v$  to a set of its indices $S$ if the vector is obtained by deleting all the elements of $v$ having indices outside of $S$. This operation in coding theory is also known as {\it puncturing}.
Contrary a vector  $v$ is {\it extended} to a vector  if the latter is obtained by appending some extra elements to $v$.

In what follows by $U(n,q)$ we denote the set of all real-valued nonzero vectors indexed by the vertices of the graph $J_q(n,1)$, $q\geq 1$ with the sum of its values being zero. 
We also use the shorthand notation $U(n)$ for $U(n,q)$  when $q$ is $1$.

\begin{theorem} \label{TFirsteigen}

1. The set $U(n,q)$ is the set of all $\theta_1(J_q(n,1))$-eigenvectors of $J_q(n,1)$.

2. \cite{D}\cite{NA} $W^T(U(n,q))$ is the set of all $\theta_1(J_q(n,k))$-eigenvectors of $J_q(n,k)$.

3.  If $S$ is a $q$-ary Steiner triple system of order $n$,  then $W_{S}^T(U(n,q))$ is the set of all $\theta_1(\Gamma_S)$-eigenvectors of its block graph $\Gamma_S$.

4. Let $S$ be a $q$-ary Steiner triple system of order $n$ and $\Gamma_S$ be its block graph. The restriction of any $\theta_1(J_q(n,3))$-eigenvector of $J_q(n,3)$ to the blocks of $S$ is a $\theta_1(\Gamma_S)$-eigenvector of $\Gamma_S$ and each $\theta_1(\Gamma_S)$-eigenvector of $\Gamma_S$ is extended to a unique $\theta_1(J_q(n,3))$-eigenvector of $J_q(n,3)$. 

5. If $S$ is a $q$-ary Steiner triple system of order $n$,  then $m(1,\Gamma_S)\leq m(1,J_q(n,[^3_1]_q))$.

\end{theorem}
\begin{proof}
1. The graph $J_q(n,1)$ is the complete graph $K_{[^n_1]_q}$. Since $\theta_1(J_q(n,1))=-1$, we see a $(-1)$-eigenvector $u$ of $K_{[^n_1]_q}$ is such that $\sum\limits_{x\in K_{[^n_1]_q}} u_x=0$ and vice versa.

2. For the Johnson graphs $J(n,k)$ this property was established by Delsarte \cite{D}. For the Grassmann graphs the proof could be found in \cite{NA}. We note that the result follows by consecutively applying Theorem \ref{TAM}.1 to the pairs of the graphs $J_q(n,i)$ and $J_q(n,i+1)$ for $i=0,\ldots,k-1$. The graphs are the halved graphs of the bipartite graph, with the adjacency being the inclusion relation of $i$-subspaces into $i+1$-subspace of $\FF_q^n$ (which is $\overline{\Gamma}$ in Theorem \ref{TAM}).

3. In view of Theorem \ref{TAM} consider the bipartite graph $\overline{\Gamma}$ where the adjacency is the containment relation for the vertices (subspaces and subsets for $q=1$) of $J_q(n,1)$ and $J_q(n,k)$.  
The graph is the same as in Example 1 but conversely to it the block intersection graph of $S$ is denoted by $\Gamma$ and the complete graph $K_{[^n_1]_q}$ is denoted by $\Gamma'$. The values 
mentioned in Theorem \ref{TAM} are $c=[^3_1]_q$, $c'=\frac{[^n_1]_q-1}{[^3_1]_q-1}$,  $m=m'=1$. By Theorem \ref{TAM}.1, we see that any vector of $W_S^T(U(n,q))$ is a eigenvector  of  $\Gamma_S$ with eigenvalue $c'-c+\theta_1(J_q(n,1))= \frac{[^{n}_1]_q-1}{[^3_1]_q-1}-[^3_1]_q-1=\theta_1(\Gamma_S)$.

 On the other hand, we apply Theorem \ref{TAM}.1 to $\overline{\Gamma}$  with the interchanged roles of $\Gamma$ and $\Gamma'$. We see that any $\theta_1(\Gamma_S)$-eigenvector $v$ of $\Gamma_S$ implies that $W_S v$ is a $\theta_1(J_q(n,1))$-eigenvector of the graph $J_q(n,1)$. Therefore, $W_S^T$ establishes an isomor\-phism between the first eigenspaces of $J_q(n,1)$ and $\Gamma_S$. 

4.
From the second and third statements of the theorem,  we see that  
$W_S^T$  ($W_{J_q(n,3)}^T$ respectively) settles an isomorphism between the first eigenspaces of the complete graph  and the block graph (the graph $J_q(n,k)$ respectively). The blocks of each STS $S$ could be treated as vertices of $J_q(n,3)$ and the rows of the block-point incidence matrix $W_S^T$ for $q$-ary STS $S$ form a subset of the rows of the block-point incidence matrix $W_{J_q(n,3)}^T$ for $J_q(n,3)$, which implies the required.

5. From the fourth statement of the current theorem we see that the restriction of a $\theta_1(J_q(n,3))$-eigenvector $v$ of $J_q(n,3)$ to the blocks of $S$ is a  $\theta_1(\Gamma_S)$-eigenvector  $v'$ of $S$.
Obviously the norm is not increased upon restriction and we have that $\|v\|_{\infty}\geq \|v'\|_{\infty}$.

\end{proof}
{\bf Remark 1.} When $q$ is $1$ (i.e. for STS in traditional sense and Johnson graphs) the bound in Theorem \ref{TFirsteigen}.5 is not sharp as we show in the next Section 
that for any STS $S$ of order $n$  $m(1,\Gamma_S)\leq 5$ whereas $m(1,J(n,3))\geq 6$ for all $n\geq 64$.

\section{Minimum of the $L_{\infty}$ norm on nowhere-zero integer eigenvector for the block graphs of STSs and Johnson graphs} 
In the rest of the paper, we set $q=1$ and consider Steiner triple systems in the classical sense.

\subsection{Minimum of the $L_{\infty}$ norm of nowhere-zero integer eigenvectors with the first eigenvalue for Johnson graphs}

In this subsection we denote by $(n,k)$ the greatest common divisor of $n$ and $k$.

\begin{lemma}\label{l:samefrac}
Any $\theta_1(J(n,k))$-eigenvector of $J(n,k)$ is equal to $W^Tu$, for some $u\in U(n)$. If $W^Tu$ is a integer $\theta_1(n,k)$-eigenvector of $J(n,k)$, then 
 $u_i$ and $u_j$ have the same fractional parts for all $1\leq i,j\leq n$. Moreover, the fractional parts equal $\frac{r}{s}$, where $r$ and $s$ are some non-negative integers such that 
$0 \leq  r < s$, $(r,s)=1$ and $s$ is a divisor of $(n,k)$.   
\end{lemma}
\begin{proof}
By Theorem \ref{TFirsteigen} we see that any  $\theta_1(J(n,k))$-eigenvector of $J(n,k)$ is equal to $W^Tu$, where  $u$ is such that $u\in U(n)$. 
  Let us prove that $u_i$ and $u_j$ have the same fractional parts for any $i \ne j$. Consider some pairwise distinct positions $i_1=i,i_2\ldots,i_k$, which are different from $j$. By hypothesis of the theorem we have $u_i+u_{i_2}+\ldots+u_{i_k}$ and $u_j+u_{i_2}+\ldots+u_{i_k}$ are integers. Hence, $u_i-u_j$ is an integer and $u_i$ and $u_j$ have the same fractional parts. Denote this fractional part by $\alpha$. Consider some $k$ elements of $u$: $u_{i_1}\ldots,u_{i_k}$. We have that the sum $u_{i_1}+\ldots+u_{i_k}$ is an integer (since $W^Tu$ is an integer vector). On the other hand, this sum has the same fractional part as $k\alpha$. Hence, $\alpha$ is rational and can be represented as $\frac{r}{s}$, where $r$ and $s$ are non-negative integers, $0 \leq r < s$,
  $(r,s)=1$. %evg added $(r,s)=1$ как в утверждении леммы
Also we have $\frac{kr}{s}$ is an integer, hence, $s$ is a divisor of $k$. Since $\sum\limits_{i=1}^n u_i=0$, $\frac{nr}{s}$ is an integer and, hence, $s$ is a divisor of $n$. Therefore, $s$ is a divisor of $(n,k)$.
\end{proof}

\begin{proposition}\label{p:samecoef}
If $n \geq 2k$ then we have $m(1,J(n,k)) \leq \frac{n-k}{(n,k)}+1$.
\end{proposition}
\begin{proof}
    For this statement we take the vector $$u^{T}=(\frac{1}{(n,k)},\ldots,\frac{1}{(n,k)},-\frac{n-1}{(n,k)}),$$
    of length $n$, where $\frac{1}{(n,k)}$ is repeated $n-1$ times. 
   The vector $W^Tu$ has two different values $\frac{k}{(n,k)}$ and $-\frac{n-k}{(n,k)}$. Since $n \geq 2k$, we have $\|W^Tu\|_{\infty}=\frac{n-k}{(n,k)}$. We see that  $\sum\limits_{i=1}^{n} u_{i}=0$ and from Lemma \ref{l:samefrac} the vector $W^Tu$ is a nowhere-zero integer $\theta_1(J(n,k))$-eigenvector of the graph $J(n,k)$.
\end{proof}

Note that in Proposition~\ref{p:samecoef} the bound tends to infinity when $n$ is growing as a function of $k$. However, for a "small"\, $n$, for example $n=2k$, it can be sharp. 
In the following statements we provide further upper bounds for odd and even cases of $k$.

\begin{proposition}\label{p:oddk}
    Let $k$ be odd. Then 

    1. If $n$ is even then $m(1,J(n,k))\leq k+1$. 

    2. If $n$ is odd then $m(1,J(n,k))\leq 2k+1$.
\end{proposition}
\begin{proof}
By Lemma \ref{l:samefrac}, the vector $W^Tu$ is a $\theta_1(J(n,k))$-eigenvector of $J(n,k)$ if $u\in U(n)$.

    1) For the first statement we take the vector 

    $$u^T=(1,\ldots,1,-1,\ldots,-1),$$
    of length $n$ with $\frac{n}{2}$ positions with the value $1$ and $\frac{n}{2}$ positions with the value $-1$. 
    %Since $k$ is odd, sum of the values of any $k$ positions not equals $0$.

    2) For the second statement we take the vector 
    $$u^T=(k+1, -1,\ldots,-1,1,\ldots,1),$$ 
    of length $n$ with $\frac{n-k-2}{2}$ positions with the value $1$, $\frac{n+k}{2}$ positions with the value $-1$, where $u_1=k+1$.  
\end{proof}

\begin{proposition}\label{p:upbound}
Let $k$ be even and $\gamma$ be the smallest positive integer number that does not divide $k$. We have the following.

1. If $n$ is divisible by $\gamma$ then $m(1,J(n,k)) \leq (\lfloor\frac{\gamma}{2}\rfloor+1)k+1$.

2. If $n$ is not divisible by $\gamma$ then $m(1,J(n,k)) \leq (\lfloor\frac{\gamma}{2}\rfloor+1)(2k+\beta-1)+1$, where $\beta$ is the remainder of division of $n-(k+1)$ by $\gamma$.
\end{proposition}
\begin{proof}
    By the hypothesis of the proposition, $n-(\beta+k+1)=q\gamma$ for some positive integers $q$ and $\beta$.
    We now define a vector $u\in U(n)$. Divide the first $n-(\beta+k+1)$ positions into $q$ %! $\gamma$ 
blocks of the same size $\gamma$.  
    If $\gamma$ is odd then each block consists of $\lfloor\frac{\gamma}{2}\rfloor$ positions with the value $\lfloor\frac{\gamma}{2}\rfloor+1$ and $\lfloor\frac{\gamma}{2}\rfloor+1$ positions with the value $-\lfloor\frac{\gamma}{2}\rfloor$.
    If $\gamma$ is even then each block consists of $\lfloor\frac{\gamma}{2}\rfloor-1$ positions with the  value $\lfloor\frac{\gamma}{2}\rfloor+1$ and $\lfloor\frac{\gamma}{2}\rfloor+1$ positions with the value $-\lfloor\frac{\gamma}{2}\rfloor+1$.
    Note that we described the values in all $\gamma$ positions in each block and the sum of the values in each block equals $0$. 
    We now define the values for the remaining $k+\beta+1$ positions of $u$. We set one element to be equal to $-(k+\beta)(\lfloor\frac{\gamma}{2}\rfloor+1)$ and the other $(k+\beta)$ elements to be $(\lfloor\frac{\gamma}{2}\rfloor+1)$. 
The sum of the values in all positions of $u$ equals $0$, i.e. $u\in U(n)$, and the sum of the values in any $k$ positions is an integer, i.e.  $W^Tu$ is an eigenvector that takes only integer values. 

Let us prove that the sum of the values in any $k$ distinct positions is not $0$. If one of the elements in these positions is $-(\beta+k)(\lfloor\frac{\gamma}{2}\rfloor+1)$ then the sum is less than $0$ because the  absolute value of any other element of $u$ is not greater than $\lfloor\frac{\gamma}{2}\rfloor+1$.
    
    Consider the case $\gamma$ is even i.e. $\gamma=2l$. Take $x$, $x \in \{0,1,\ldots,k\}$ positions with the values $l+1$ and $k-x$ positions with the values $-l+1$.
    The sum of the values in these positions equals $lx+x-kl+k+lx-x=2lx-k(l-1)$ which is $0$ if and only if $$2lx=k(l-1).$$ 

Let $l$ be odd. From the equality above we see that $k$ is divisible by odd $l$ and by $2$ from the condition of the proposition, so $k$ is divisible by $\gamma=2l$, which contradicts the definition of $\gamma$. If $l$ is even, then $l-1$ is odd. Hence, again from the equality above we see that $k$ should be divisible by $\gamma=2l$, a contradiction. 
    
    Consider the case when $\gamma$ is odd, i.e. $\gamma=2l+1$.
    Take $x$ positions, $x \in \{0,1,\ldots,k\}$ with the  value $l+1$ and $k-x$ positions with the value $-l$. 
    The sum of the values in these positions equals $lx+x-kl+lx$. If this sum equals $0$ then $x(2l+1)=kl$. Since $k$ is not divisible by $\gamma=2l+1$, this equality does not hold. %!this sum is equal to 0

    Therefore, $W^Tu$ is a NZI vector. If $n$ is divisible by $\gamma$, then $\|W^Tu\|_{\infty}\leq(\lfloor\frac{\gamma}{2}\rfloor+1)k$. If $n$ is not divisible by $\gamma$, then $\|W^Tu\|_{\infty}\leq(\lfloor\frac{\gamma}{2}\rfloor+1)(2k+\beta-1)$.
    %evg в этом абзаце оба равентсва заменил на неравенства. В первом случае при небольших $n$ может не найтись $k$ положительных чисел, во втором случае у нас отрицательные значения меньше по абсолютной величине, поэтому  столько не наберется. 
\end{proof}

In the following lemmas we study the structural properties of NZI vectors as we are working towards lower bounds on $m(1,J(n,k))$.

\begin{lemma}\label{l:notzerosum}
    Let a vector $u\in U(n)$ have at least $k$ positions with the value $a+\frac{r}{s}$ and at least $k$ positions with the value $-b+\frac{r}{s}$ for some non-negative integers $a,b, r$ and $s$, where $b \ne 0, 0 \leq r < s, (r,s)=1 \text{ and } s \text { is a divisor of } (n,k)$. If $W^{T}u$ is a NZI vector then the number $\frac{k(bs-r)}{s(a+b)}$ is not an integer.
\end{lemma}
\begin{proof}
    Take $x$, $0 \leq x \leq k$, positions of $u$ with the value $a+\frac{r}{s}$ and $k-x$ positions with the value $-b+\frac{r}{s}$. The sum of these values equals $x(a+b)+\frac{kr}{s}-bk$. This sum equals $0$ if and only if $x=\frac{(bs-r)k}{s(a+b)}$ is an integer.
\end{proof}

\begin{lemma}\label{l:eqval}
Let $u$ be a vector in $U(n)$ such that $W^Tu$ is NZI and $\|W^Tu\|_{\infty}+1=m(1,J(n,k))$. Let  $n>j^2+2kj +3k-j-\frac{j^2}{k}$, where $j=2k$ if $k$ is odd and $j=(\lfloor\frac{\gamma}{2}\rfloor+1)(2k+\gamma)$ if $k$ is even, $\gamma$ be the smallest positive integer number that is not a divisor of $k$. Then $u$ has at least $k$ positions with the same positive value and at least $k$ positions with the same negative value.  
\end{lemma}
\begin{proof}
 If vector $u$ has $k$ positions with the value $0$ then an element of $W^Tu$ is zero. Hence, at least $n-k+1$ positions of $u$ has nonzero values.
 By Propositions ~\ref{p:oddk} and ~\ref{p:upbound} we have that $\|W^Tu\|_{\infty} \leq j$, where $j=2k$ if $k$ is odd and $j= (\lfloor\frac{\gamma}{2}\rfloor+1)(2k+\gamma)$ if $k$ is even. Here $\gamma$ is the smallest positive integer number that does not divide $k$.
 %Denote $m=\frac{(\lfloor\frac{\gamma}{2}\rfloor+1)(2k+\beta-1)}{k}$. 
 Also denote by $x$ the number of positions with positive values and  by $y$ the number of positions with negative values. By the remark on the number of nonzero elements in $u$ in the beginning of the proof,  we have that
\begin{equation}\label{exy}x+y\geq n-k+1.\end{equation}%!added more details
 
 Let us denote $\frac{j}{k}$ by $m$. By Lemma~\ref{l:samefrac} the values of the vector $u$ in all positions have the same fractional part. If $u$ does not have $k$ positions with the same value then there are not more than $k-1$ positions with the integer part $i$ for any nonnegative $i$. So there are not more than $(\lfloor m \rfloor+1)(k-1)$ positions with the integer part not more than $\lfloor m \rfloor$. 
 So, if $x>(m+2)(k-1) \geq (\lfloor m \rfloor+2)(k-1)$ %! it was $x>(m+2)(k-1) \geq \lfloor m \rfloor(k-1)$ before  %evg тут (\lfloor m \rfloor+2) все таки а не (\lfloor m \rfloor+1) 
  %evg пробел добавил перед then
 then there are at least $k$ positions with the same positive value or there are $k$ positions with values that are more than $\lfloor m \rfloor +1 > m$. In the latter case we have	
 $\|W^Tu\|_{\infty}>mk=j$, a contradiction. 
 Analogously, if $y>(m+2)(k-1)$ then there are at least $k$ positions with the same negative value or we have a contradiction with the minimality of norm. 
 
We show that  $n-k+1>2(m+2)(k-1)$.
By the condition of the lemma we have the following:

$$n-k+1-2(m+2)(k-1)=n-k+1-2(\frac{j}{k}+2)(k-1)>$$
$$j^2+2kj+3k-j-\frac{j^2}{k}-k+1-2j-4k+\frac{2j}{k}+4=$$

$$j^2-3j-2k+5+2kj-\frac{j^2}{k}+\frac{2j}{k}=$$

$$j(j(1-\frac{1}{k})-3+2k+\frac{2}{k})+5-2k $$

Note that $j\geq 2k\geq 4$, so 
$$j(j(1-\frac{1}{k})-3+2k+\frac{2}{k})+5-2k\geq$$

$$j(2k-2-3+2k+\frac{2}{k})+5-2k\geq j-2k+5>0.$$

 %! I added the inequalities above to explain why $n-k+1>2(m+2)(k-1)$, may be a simplier argument  is possible

Because $n-k+1>2(m+2)(k-1)$ then $x$ or $y$ is more than $(m+2)(k-1)$. 
 Without loss of generality, we assume that $x>(m+2)(k-1)$. 
 If $y$ is also more than $(m+2)(k-1)$, then the Lemma holds, so we consider the case \begin{equation}\label{ee1}y \leq (m+2)(k-1)\end{equation} in more details below.

 In view of Lemma \ref{l:samefrac},  we have that the positive values of $u$ are not less than $\frac{1}{k}$. 
This, combined with inequalities (\ref{exy}) and (\ref{ee1}) implies that the sum of all positive elements in $u$ is such that \begin{equation}\label{ee2}\sum\limits_{i=1,\ldots,n,u_i>0} u_i\geq \frac{x}{k} \geq \frac{n-k+1-y}{k} \geq \frac{n-k+1-(m+2)(k-1)}{k}.\end{equation} %% Mb switch to equation format?? 

Case 1.  If $y<k$ we consider the sum of $y$ positions with negative values and $k-y$ positions with the smallest positive values in the vector $u$. Since $\sum\limits_{i=1}^{n} u_i=0$ the absolute value of this sum  equals the sum of $x-k+y$ positions with the largest positive values. Since each positive value in $u$ is at least $\frac{1}{k}$, this sum is not less than $\frac{x-k+y}{k} \geq \frac{(n-k+1-y)-k+y}{k}=\frac{n-2k+1}{k}$. Hence, because $n>j^2+2kj +3k-j-\frac{j^2}{k}>jk+2k+1$, which we have by the hypothesis of the Lemma, we obtain 
$$\|W^Tu\|_{\infty}>\frac{n-2k+1}{k}>\frac{jk+2k+1-2k+1}{k}=\frac{jk+2}{k}>j,$$%!gave more details here, please crosscheck
 which is a contradiction.
  
Case 2. If $y\geq k$ we
 consider the $k$ minimum negative values of $u$.
 The sum of these values is not greater than average negative value $k$ times. 
 Because $\sum\limits_{i=1,\ldots,k}u_i=0$, the absolutes of sum of negatives and sum of positives coincide.
 %evg Здесь было Because $\sum\limits_{i=1,\ldots,k}u_i=0$, the absolutes of average negative and average positive coincide. Это вообще говоря неверно, тут наверно имелось ввиду все-таки суммы, а не средние, поменял на суммы  Дальше вроде бы правильно
So, from (\ref{ee2}) the average 
among
%evg здесь было amoung, вроде among правильно
all negative is less then or equal to $-\frac{(n-k+1-y)}{ky}$ and there are pairwise distinct $i_1,\ldots,i_k$:

\begin{equation}\label{ee3} \sum\limits_{l=1,\ldots,k} u_{i_l}\leq -\frac{(n-k+1-y)}{y}.\end{equation} 

We show that  we have $n>(j+1)y+k-1$.
By the hypothesis of the Lemma

$$n>j^2+2kj +3k-j-\frac{j^2}{k}=j(j+2k-1-\frac{j}{k})+3k=j(mk+2k-1-m)+3k=$$
$$j(mk+2k-2-m)+mk+3k>j(mk+2k-2-m)+(mk+2k-2-m)+k=$$
$$(j+1)(m+2)(k-1)+k$$
From (\ref{ee1}) we have that 
$$(j+1)(m+2)(k-1)+k\geq (j+1)y+k> y(j+1)+k-1,$$
%% again, changed ">" to  "$\geq$" 
so we have that $n>(j+1)y+k-1$ and therefore from $\sum\limits_{l=1,\ldots,k} u_{i_l}\leq -\frac{(n-k+1-y)}{y} <-j$, which contradicts $\|W^Tu\|_{\infty}\leq j$.

\end{proof}

We introduce extra notations. Let $\gamma$ be  the smallest positive integer number that is not a divisor of  $k$.
Denote by $T(k)$ the number 
$j^2+2kj +3k-j-\frac{j^2}{k}$,  where  $j=(\lfloor\frac{\gamma}{2}\rfloor+1)(2k+\gamma)$ if $\gamma$ is even and $j=2k$ if $\gamma$  is odd. 

We also introduce several extra sets and a number:
\[
\begin{split}
   & B(n,k)=\{(a,b,r,s): \frac{(bs-r)k}{s(a+b)} \text{ is not an integer},  
    \text{ where } a,b,r,s\geq 0  \text{ are integers, }\\
 &b\ne 0,\text{ }0 \leq r < s, (s,r)=1, s \text { is a divisor of } (n,k) \}, 
    \end{split}
    \] 

\[
\begin{split}
    N(n,k)=\min\{\max\{k(a+\frac{r}{s}), k(b-\frac{r}{s})\}: (a,b,r,s) \in B(n,k)\}
    \end{split}
    ,\] 
\[
\begin{split} 
M(n,k)=\{(a,b,r,s) \in B(n,k): \max\{k(a+\frac{r}{s}), k(b-\frac{r}{s})\}=N(n,k) \}.
\end{split}
\]

\begin{theorem}\label{t:bound}
    If  $n>T(k)$, then $m(1,J(n,k)) \geq N(n,k)+1$.  
\end{theorem}
\begin{proof}
   In view of Theorem \ref{TFirsteigen} consider the vector $u$, $u\in U(n)$ such that $W^Tu$ is a NZI $\theta_1(J(n,k))$-eigenvector of $J(n,k)$ and $\|W^Tu\|_{\infty}+1=m(1,J(n,k))$. By Lemma~\ref{l:eqval} we have that there are $k$ positions of $u$ with the value $a+\frac{r}{s}$ and $k$ positions with the value $-b+\frac{r}{s}$ for some $b\ne 0, (r,s)=1, 0 \leq r < s, s \text { is a divisor of } (n,k)$. 
    By Lemma~\ref{l:notzerosum} we have that $\frac{(bs-r)k}{s(a+b)}$ is not an integer. 
    Therefore, $(a,b,r,s) \in B(n,k)$. On the other hand, $\|W^Tu\|_{\infty}$ is not less than $\max\{k(a+\frac{r}{s}),k(b-\frac{r}{s})\}$, hence, $\|W^Tu\|_{\infty} \geq N(n,k)$.
\end{proof}

\begin{corollary}\label{c:shbound} 
Let $(a,b,r,s)$ be in $M(n,k)$. If $n>T(k)$ and $\frac{(bs-r)n}{s(a+b)}$ is an integer, then $m(1,J(n,k))=N(n,k)+1$.
\end{corollary}
\begin{proof}
    Take the vector $u$ such that it has $\frac{(bs-r)n}{s(a+b)}$ positions with the value $(a+\frac{r}{s})$ and $\frac{(as+r)n}{s(a+b)}$ positions with the value $(-b+\frac{r}{s})$.%!$(b-\frac{r}{s})->(-b+\frac{r}{s}),
% b-\frac{r}{s} cannot be negative

 The norm of the vector $W^Tu$ equals $\max\{k(a+\frac{r}{s}),k(b-\frac{r}{s})\}=N(n,k)$. %! the same mistake here
    On the other hand,
    by Theorem~\ref{t:bound} we have $m(1,J(n,k)) \geq N(n,k)+1$, hence, $m(1,J(n,k)) = N(n,k)+1$
\end{proof}

\begin{corollary}\label{c:coprime}
    Let $\gamma$ be the smallest positive integer that does not divide $k$, $n>T(k)$, $(n,k)=1$ and $n$ be divisible by $\gamma$. We have that
    
 1. If $\gamma=2$ then $m(1,J(n,k))=k+1$.  
 
 2. If $\gamma>2$ then $m(1,J(n,k))=(\lfloor\frac{\gamma}{2}\rfloor+1)k+1$. 
\end{corollary} 
\begin{proof}
    Let $u$ be a vector in $U(n)$ such  that $\|W^Tu\|_{\infty}+1=m(1,J(n,k))$.  If $\gamma>2$ then since $(n,k)=1$, $k$ is even and $n$ is odd, so due to Proposition~\ref{p:upbound} we have $m(1,J(n,k)) \leq (\lfloor\frac{\gamma}{2}\rfloor+1)k+1$. If $\gamma=2$  then  $k$ is odd and by condition of corollary $n$ is even, so $m(1,J(n,k)) \leq k+1$ by Proposition~\ref{p:oddk}. %!\|W^Tu\|_{\infty}-> m(1,J(n,k)), because m(1,J(n,k))=\|W^Tu\|_{\infty}+1
 Since $(n,k)=1$ the fractional part of the value in any position of $u$ equals $0$ and any $(a,b,r,s)$ in $M(n,k)$ has $r=0, s=1$. Let  $(a,b,r,s)$ be an quadruple from $M(n,k)$.
 We have that $\frac{bk}{a+b}$ is not an integer and $(a+b)$ is not a divisor of $k$. So we have that $a+b \geq \gamma$. Note that if $a+b=\gamma$, then $(a,b)=1$.
Indeed otherwise we have $\frac{a+b}{(a,b)}<\gamma$, so $\frac{bk}{a+b}$ is integer. We obtain that $(a,b,0,1)\notin B(n,k)$, which is a contradiction.

 On the other hand, if $a+b=\gamma$, then $\max\{a,b\} \geq \lfloor\frac{\gamma}{2}\rfloor+1$ (if $\gamma=2$ then $\max\{a,b\}=1$). If $a+b>\gamma$ then $\max\{a,b\}$ is also not less than $\lfloor\frac{\gamma}{2}\rfloor+1$ ($1$ in the case $\gamma=2$). These lower bounds hold for any $(a,b,0,1)\in M(n,k)$ and therefore $N(n,k) \geq (\lfloor\frac{\gamma}{2}\rfloor+1)k$ ($k$ if $\gamma=2$) and, hence, $m(1,J(n,k))=(\lfloor\frac{\gamma}{2}\rfloor+1)k+1$ ($k+1$ if $\gamma=2$).
\end{proof}

\begin{corollary}\label{c:lowerbound} %! added this corollary 
If $n>T(k)$ then $m(1,J(n,k))\geq k+1$.
\end{corollary}
\begin{proof}
Suppose that $m(1,J(n,k))<k+1$. 
Then by Theorem \ref{t:bound} we have 
$$N(n,k)\leq m(1,J(n,k))-1 <k,$$
i.e. there are integers $a\geq 0$, $b\geq 1$, $r\geq 0$ and a divisor $s$ of $(n,k)$  such that 

$$\frac{(bs-r)k}{s(a+b)} \text{ is not an integer,}$$ $$ max\{k(a+\frac{r}{s}), k(b-\frac{r}{s})\}<k.$$

The latter inequality implies that $a=0$ and $b=1$, which, combined with the fact that $s$ is a divisor of $k$ implies that
$\frac{(bs-r)k}{s(a+b)}=\frac{(bs-r)k}{s}$ is an integer, a contradiction.

\end{proof}

\begin{proposition}\label{p:nkodd}
    If $n$ is even, $k$ is odd and $n>T(k)$, then $m(1,J(n,k))=k+1$.
\end{proposition}
\begin{proof}
    From Proposition~\ref{p:oddk} we have $m(1,J(n,k)) \leq k+1$ and from Corollary \ref{c:lowerbound} we have $m(1,J(n,k)) \geq k+1$.
\end{proof}

\begin{theorem}\label{p:jn3}
  If  $n>T(3)=63$, then we have that
  \[m(1,J(n,3))=
  \begin{cases}
      4,{ } n \text{ is even}\\
      6,{ } n \text{ is odd and } n=0,6\mbox{ mod }9\\
      7,{ } \text{otherwise}.
  \end{cases}
  \]
\end{theorem}
\begin{proof}
Let $u$ be a  vector in $U(n)$ such that $\|W^Tu\|_{\infty}+1=m(1,J(n,k))$. 

The case of even $n$ is a particular case of Proposition~\ref{p:nkodd}. In what follows we assume that $n$ is odd. By Proposition~\ref{p:oddk} we have $\|W^Tu\|_{\infty} \leq 6$. 

1. Suppose the elements of $u$ are integers.
By Lemma~\ref{l:eqval} there are at least $3$ elements of $u$ equal to $a$ and at least $3$ elements equal to $-b$. If $a \geq 2$ or $b \geq 2$ then we obtain desired inequality $\|W^Tu\|_{\infty} \geq 6$. Therefore, we have  $a=1$ and $b=1$. Since $n$ is odd and $\sum\limits_{i=1}^n u_i=0$, there is an element of $u$ with an even value $c$. This value cannot be $0,2,-2$, otherwise
 we have three elements in $u$ with zero sum and the vector $W^Tu$ has zero elements. Hence, $c \geq 4$ or $c \leq -4$. In both cases, the bound $\|W^Tu\|_{\infty} \geq 6$ holds.
 %and, hence, $m(1,J(n,k)) \geq 7$. In the other hand, by Proposition~\ref{p:oddk} we have $m(1,J(n,k)) \leq 7$. 

2. Suppose the elements of $u$ have nonzero fractional parts. By Lemma \ref{l:samefrac} this fractional part is $\frac{1}{3}$ or $\frac{2}{3}$ and we necessarily have that $n$ is divisible by $3$. Without loss of generality, we can assume that this fractional part equals $\frac{2}{3}$ (otherwise we take the vector $-u$). By Lemma~\ref{l:eqval} the vector $u$ has $3$ elements with some positive value $a+\frac{2}{3}$ and $3$ elements  with some negative value $-b+\frac{2}{3}$. 
If $a=0$, then $b$ cannot be $1$ or $2$, otherwise there are three elements of $u$ with the sum equals $0$. If $b \geq 3$, then $\|W^Tu\|_{\infty} \geq 7$.
If $a \geq 2$, then $\|W^Tu\|_{\infty} \geq 8$. Both cases contradict $\|W^Tu\|_{\infty} \leq 6$. So we have $a=1$ and hence, $\|W^Tu\|_{\infty} \geq 5$. Also we have $b=1$ or $b=2$, otherwise $\|W^Tu\|_{\infty} \geq 7$.

2a. Let $n$ be $3 \mod 9$. We also recall that $n$ is odd and there are at least $3$ elements in $u$ equal $\frac{5}{3}$ and at least $3$ elements equal $-b+\frac{2}{3}$, where $b$ is $1$ or $2$. We show that $\|W^Tu\|_{\infty} \geq 6$.

Consider case $b=1$. 
Since $\sum\limits_{i=1}\limits^{n} u_i=0$ and $n$ is odd, there is an element $\frac{c}{3}$ of $u$ for some even $c$. If $c \geq 8$ or $c \leq -16$ then $\|W^Tu\|_{\infty} \geq 6$. If $c=2,-4,-10$ then there are three elements in $u$ having zero sum, a contradiction.  

Consider case $b=2$. 
If all elements of $u$ are $\frac{5}{3}$ or $-\frac{4}{3}$ then $n$ is divisible by $9$, which is not the case. So there is at least one element of $u$ that differs from $\frac{5}{3}$ or $-\frac{4}{3}$. If there is element that is not less than $\frac{8}{3}$ or element that is not more than $-\frac{10}{3}$ or two elements that equals $-\frac{7}{3}$ then $\|W^Tu\|_{\infty} \geq 6$. 
If in vector $u$ there is an element $-\frac{1}{3}$ or two elements equals $\frac{2}{3}$ or two elements with the values $\frac{2}{3}$ and $-\frac{7}{3}$ then there are three elements in $u$ with zero sum. It remains to consider the case when one element of $u$ equals $\frac{c}{3}$, where $c$ is $2$ or $-7$, any other element is $\frac{5}{3}$ or $-\frac{4}{3}$. 
%If there is element not less than $\frac{8}{3}$ or not more than $-\frac{10}{3}$ then $\|W^Tu\|_{\infty} \geq 6$. So any such element equals $\frac{2}{3}$ or $-\frac{7}{3}$. 
Denote by $x$ the number of elements $\frac{5}{3}$ and by $y$ the number of elements $-\frac{4}{3}$ in $u$. Then the sum of all elements equals $\frac{5x-4y+c}{3}$, where $x+y=n-1$. Since $\sum\limits_{i=1}^n u_i=0$ we have $x=\frac{4n-4-c}{9}$, i.e. $x=\frac{4n-6}{9}$ or $x=\frac{4n+3}{9}$. In both cases we have $n=6 \mod 9$, a contradiction. We conclude that $\|W^Tu\|_{\infty} \geq 6$ for the case when $n=3 \mod 9$.

2b. If $n=0\mbox{ mod }9$ take the vector 
$$u^T=(\frac{5}{3},\ldots,\frac{5}{3},-\frac{4}{3}, \ldots, -\frac{4}{3}),$$
of length $n$, where $\frac{5}{3}$ is repeated $\frac{4n}{9}$ times. Hence, $\|W^Tu\|_{\infty}=5$.

2c. If $n=6\mbox{ mod }9$ take vector 
$$u^T=(\frac{2}{3},\frac{5}{3},\ldots,\frac{5}{3},-\frac{4}{3}, \ldots, -\frac{4}{3}),$$
of length $n$, where $\frac{5}{3}$ is repeated $\frac{4n-6}{9}$ times, $-\frac{4}{3}$ is repeated $\frac{5n-3}{9}$ times and $u_1=\frac{2}{3}$. Hence, $\|W^Tu\|_{\infty}=5$.

\end{proof}

\subsection{Minimum of $L_{\infty}$ norm of nowhere-zero integer eigenvectors with the first eigenvalue for block graphs of STSs}

\begin{theorem}
Let $S$ be a Steiner triple system of order $n$, $n>7$. If $n=1 \mbox{ mod } 4$ then $m(1,\Gamma_S)\leq 4$  and if  $n=3 \mbox{ mod } 4$ then
$m(1,\Gamma_S)\leq 5$.
\end{theorem}
\begin{proof}

Let $n$ be  such that $n=1\mbox{ mod } 4$. Suppose $S$ is a STS with all triples containing the point $1$ being $\{1,2,3\},\ldots,\{1,n-1,n\}$.
We set the vector $u$ as follows: $u_1=0$, $u_i=1$ for $2\leq i\leq \frac{n+1}{2}$ and $u_i=-1$ for $\frac{n+3}{2}\leq i\leq n$. The vector  $W_S^Tu$ is a $\theta_1(\Gamma_S)$-eigenvector of $\Gamma_S$ due to Theorem \ref{TFirsteigen}.3. The elements of $W_S^Tu$ are $u_i+u_j+u_k$ if $\{i,j,k\}$ is a triple of $S$. If $\{i,j,k\}$ contains $1$, then by the choice of u, we see that $(W_S^Tu)_{\{i,j,k\}}$ is either $2$ or $-2$.
Otherwise, $u_i+u_j+u_k$ is the sum of three numbers with absolute  value $1$, therefore $\|W_S^Tu\|_{\infty}\leq 3$.

Let $n$ be such that $n=3\mbox{ mod } 4$. Without restriction of generality,   $\{1,2,3\}$ is a triple in $S$. Set the vector $u$ as follows: $u_1=-1$, $u_2=2$, $u_3=-3$. We will arrange the remaining $\frac{n-5}{2}$ values $-1$ and $\frac{n-1}{2}$ values $1$ in the remaining $n$ positions of $u$ according to the structure of $S$. 

 We consider an auxiliary graph on the vertices, which are the points $\{4,\ldots,n\}$. The edges are the pairs obtained from the all triples of $S$ containing $2$ or $3$, excluding $\{1,2,3\}$, by removing the points $2$ and $3$. The edges are labeled "2"\,or "3"\,which is the point that completes the edge, i.e. pair of points, to a triple of $S$. From the definition of a Steiner triple system, we see that the graph is the union of even length cycles that partition $\{4,\ldots,n\}$, where labels, i.e. 2 and 3, for any two incident edges are different.

For a vector $u$ we set its remaining elements (with indices of the vertices of the auxiliary graph) to $1$ and $-1$ as follows. We distinguish one cycle $i_1,\ldots,i_{2l}$. Choose a path $i_1$, $i_2$, $i_3$ in the cycle and set values $u_{i_1}=u_{i_2}=u_{i_3}=1$. 
The remaining values of $u$ in this cycle are alternating $-1$ and $1$: $u_{i_4}=-1$, $u_{i_5}=1$, $u_{i_6}=-1$,\ldots, $u_{i_{2l}}=-1$.
  For any other cycle $j_1,\ldots,j_{2m}$ we set the values of $u$ in the alternating way: $u_{j_{2s+1}}=-u_{j_{2(s+1)}}=1$, for $s=0,\ldots, m-1$.
By the choice of the values of $u$ on the points $\{4,\ldots,n\}$, for any two adjacent vertices $a$, $b$ in the auxiliary graph  we have that  $u_a+u_b=0$ or $2$.

Since $\sum\limits_{i=1}^{n}u_i=0$, $u$ is a (-1)-eigenvector of $K_n$. Now consider the vector $W^Tu$, which is a $\theta_1(\Gamma_S)$-eigenvector of $\Gamma_S$ by Theorem \ref{TFirsteigen}.3.
For a triple $\{i,j,k\}$ of $S$ such that $\{i,j,k\}\cap \{2,3\}=\emptyset$,  we see that $u_i+u_j+u_k$  is the sum of elements whose absolute values are $1$ (note that $u_1=1$ and $|u_i|=1,i=4\ldots,n$). So $u_i+u_j+u_k$ is nonzero with the absolute value less or equal to $3$.
Let $\{i,j,k\}$ be a triple of $S$, such that $k=2$ or $3$ and $i, j\geq 4$.  
Since $u_2=2$, $u_3=-3$ and $u_i+u_j$ is either $0$ or $2$, the sum $u_k+u_i+u_j$ is nonzero with the absolute value less or equal to $4$. 
For the last remaining case of triple $\{1,2,3\}$ in $S$ we have $u_1+u_2+u_3=-2$.
We see that $\|W_S^T u\|_{\infty}\leq 4$.

\end{proof}

\section{Nowhere zero flows for families of classic Steiner triple system}
Throughout this section, we use terms of flows in Steiner triple systems in the sense of work \cite{Akbari} rather than terms of eigenvectors.

\subsection{Flows for some classic STSs}

\begin{lemma}\label{lemmaresolve}\cite[Lemma 1.4]{Akbari2017} Let $S$ be a resolvable STS of order $n$. If $n=1\mbox{ mod }4$, $S$ has a $2$-flow, otherwise it has a $3$-flow.
\end{lemma}
%\begin{proof} If  $n=1\mbox{ mod }4$, we split the $\frac{n-1}{2}$ parallel classes into two equal sized subsets and set $v_T=1$ if $T$ belongs to any parallel class from the first subset and %$-1$ otherwise.
%It is easy to see that $v$ is $2$-flow for $S$.

%If  $n=3\mbox{ mod }4$, we part the $\frac{n-1}{2}$ parallel classes into three subsets  sets of sizes $\frac{n-7}{4}$ , $\frac{n-3}{4}$ and $2$ and set the elements of vector $v$ to be $2$, %$-2$ and $1$ if the triple belongs to parallel classes in these sets respectively.
%It is easy to see that $v$ is $3$-flow for $S$.

%\end{proof}

\begin{proposition}
1. Let $S$ be the Steiner triple system formed by the supports of weight three codewords of the Hamming code of length $2^r-1$, $r\geq 4$. Then $S$ admits a $3$-flow, but does not admit a $2$-flow.

2. Let $S$ be the original Bose Steiner triple system of order $3p$ constructed from the latin square of $Z_{p}$, where $p$ is an odd prime. Then $S$ has a $2$- or a $3$-flow. 
\end{proposition}
\begin{proof}
1. If $r$ is even, then the Hamming STS of order $2^r-1$ is known to be resolvable, i.e. its blocks are parted into $\frac{2^r-2}{2}$ parallel classes \cite{Baker}.  
In this case, following Lemma \ref{lemmaresolve},  we see that a $3$-flow exists.
 
When $r$ is odd, then it is easy to see that $2^r$ is $2$ modulo $6$.
Any Hamming STS is cyclic, i.e. it has an automorphism of order being equal to the order of $S$. 
 By \cite[Theorem 3.6]{AMMW} if the order of a STS $S$ is $1$ modulo $6$ and $S$ is cyclic, then a $3$-flow for $S$ exists.  

Suppose a $2$-flow $v$ for Hamming STS $S$ exists, so $W_Sv=0$ and the elements of the $v$ are $+1$ and $-1$. This contradicts the fact that any row of $W_S$ has exactly  $\frac{n-1}{2}=\frac{2^r-1-1}{2}$ of ones, which is odd.

2.  The Steiner triple systems of order $9p$  constructed by 	Bose method   were recently shown to be resolvable by Colbourn and Lusi \cite{ColbournLusi}. The result follows from Lemma \ref{lemmaresolve}. 
\end{proof}

\subsection{Flows for Assmuss-Mattson construction}\label{S:AM}

Let us consider the Assmus-Mattson construction \cite{AM}.
Given a Steiner triple system $S$  of order $n$ with the pointset $\{1,\ldots,n\}$ and a point $i$, we denote by $\bar{i}$ the number $i+n$. For a function $\tau:S\rightarrow \{0,1\}$, we define the Assmuss-Mattson Steiner triple system of order $2n+1$ with point set $\{1,2,\dots, n, \bar{1},\bar{2}, \dots, \bar{n}, 2n+1\}$:

$${\bar S}=\bigcup_{\{i,j,k\}\in S}{P(\{i,j,k\})}\cup$$ 
$$\bigcup_{i=1,\dots, n}{\{i,\bar{i},2n+1\}},$$
where $$P(\{i,j,k\})=\{\{i,j,k\},\{i,\bar{j},\bar{k}\},\{\bar{i},j,\bar{k}\},\{\bar{i},\bar{j},k\}\},$$
if $\tau(\{i,j,k\})=0$ and  
$$P(\{i,j,k\})=\{\{\bar{i},j,k\},\{\bar{i},\bar{j},\bar{k}\},\{i,j,\bar{k}\},\{i,\bar{j},k\}\}$$
 otherwise.

Our goal is to construct a zero-sum $5$-flow for $\bar{S}$, i.e. to find a NZI vector $v$ such that $\sum\limits_{T\in \bar{S}, i\in T}{v_T}=0$ for any $i\in \{1,2,\dots, n, \bar{1},\bar{2}, \dots, \bar{n}, 2n+1\}$
and $\|v\|_{\infty}=4$.

For our following arguments we need several auxiliary statements.
 
	Let $a_1,a_2,a_3$ be pairwise distinct elements from $\{1,\dots,n\}$. We define a real-valued vector $g$ indexed by the triples in $P(\{a_1,a_2,a_3\})$ depending on $\tau(\{a_1,a_2,a_3\})$.	
	If  $\tau(\{a_1,a_2,a_3\})=0$, define the elements of $g$ (indexed by triples in $P$) as follows:  
	$$g_{\{a_1, a_2, a_3\}}=g_{\{a_1, \bar{a}_2, \bar{a}_3\}}=1, g_{\{\bar{a}_1, \bar{a}_2, a_3\}}=g_{\{\bar{a}_1, a_2, \bar{a}_3\}}=-1.$$ If $\tau(\{a_1,a_2,a_3\})=1$ define  $$g_{\{\bar{a}_1, a_2, a_3\}}=g_{\{\bar{a}_1, \bar{a}_2, \bar{a}_3\}}=-1, g_{\{a_1, a_2, \bar{a}_3\}}=g_{\{a_1, \bar{a}_2, a_3\}}=1.$$  

Consider the properties of  the introduced vector.

\begin{proposition}\label{P:zero} 
 The following holds for vector $g$:
 
\begin{equation}\label{eqgvec11}
\sum_{T'\in P(\{a_1,a_2,a_3\}), a_1\in T'}{g_{T'}}=2,
\end{equation}

\begin{equation}\label{eqgvec12}
\sum_{T'\in P(\{a_1,a_2,a_3\}), \bar{a}_1\in T'}{g_{T'}}=-2,\end{equation}

\begin{equation}\label{eqgvec2}\sum_{T'\in P(\{a_1,a_2,a_3\}), s\in T'}{g_{T'}}=0\mbox{ for }s\in \{a_2, a_3, \bar{a}_2, \bar{a}_3\}.\end{equation}

\end{proposition}
\begin{proof}
	The proof is obtained by direct calculations. 
\end{proof}
Clearly, by direct permutation of points $a_1$, $a_2$ and $a_3$ one may apply Proposition \ref{P:zero} for $a_2$ or $a_3$ instead of $a_1$.

\begin{lemma}\label{L:4covering} 
	Let $S$ be $STS(n)$, where $n\geq 49$. Then there is a function $h:S\rightarrow \{1,\ldots,n\}$ such that
	\begin{enumerate}
		\item for any $\{i,j,k\} \in S$, $h(\{i,j,k\})\in \{i,j,k\}$,
		\item for any $i\in \{1,\ldots,n\}$, $|\{T\in S: h(T)=i\}|\geq 4$.
	\end{enumerate}
		    
\end{lemma}
\begin{proof}
	Consider the set $A$ of all functions satisfying the first condition from Lemma. Each function maps any triple $T\in S$ to a point in $T$. Therefore we have $|A|=3^{\frac{n(n-1)}{6}}$. The next step is counting the number of functions from $A$ not satisfying the second condition of Lemma for a fixed point $i_0$. Clearly, for any such function $f$, $|\{T\in S: h(T)=i_0\}|$ equals $0$, $1$, $2$ or $3$. We conclude that the number of such functions is equal to
	$$R=2^{\frac{n-1}{2}}3^{\frac{n(n-1)}{6}-\frac{n-1}{2}}+{\frac{n-1}{2} \choose 1}2^{\frac{n-1}{2}-1}3^{\frac{n(n-1)}{6}-\frac{n-1}{2}}+$$ $${\frac{n-1}{2} \choose 2}2^{\frac{n-1}{2}-2}3^{\frac{n(n-1)}{6}-\frac{n-1}{2}}+{\frac{n-1}{2} \choose 3}2^{\frac{n-1}{2}-3}3^{\frac{n(n-1)}{6}-\frac{n-1}{2}}.$$   
	Consequently, the number of functions in $A$ not satisfying the second condition from Lemma (in at least one point) is not greater than $nR$. As a result, we have that the number of functions from $A$ satisfying both conditions is at least $3^{\frac{n(n-1)}{6}}-nR$ which is equal to
	%$$2^{\frac{n-1}{2}}3^{\frac{n(n-1)}{6}-\frac{n-1}{2}}\bigl((\frac{3}{2})^{\frac{n-1}{2}}-n(1+\frac{1}{2}{\frac{n-1}{2} \choose 1}+\frac{1}{4}{\frac{n-1}{2} \choose 2}+\frac{1}{8}{\frac{n-1}{2} \choose 3}\bigr).$$    
	$$2^{\frac{n-1}{2}}3^{\frac{n(n-1)}{6}-\frac{n-1}{2}}\bigl((\frac{3}{2})^{\frac{n-1}{2}}-\frac{1}{384}n^4-\frac{1}{128}n^3-\frac{71}{384}n^2-\frac{103}{128}n\bigr).$$ 
	This expression is strictly positive for $n\geq 49$ and the proof is finished. 
\end{proof}

\begin{theorem}\label{T:AM} 
	Any Assmuss-Mattson Steiner triple system of order $N$, $N\geq 99$ admits a zero-sum $5$-flow.    
\end{theorem}

\begin{proof}
Let $\overline{S}$ of order $N=2n+1$ be obtained from a STS $S$ of order $n$, $n\geq 49$ by Assmuss-Mattson construction with a function $\tau:S\rightarrow \{0,1\}$.

	We start from the all-zero vector $v$, indexed by the triples of ${\bar S}$ and update it with the course of the proof. At the end of the proof, $v$ will be $5$-flow for STS ${\bar S}$.  	
	Consider the set $B=\{T\in S: 1\in T\}$ and $T_0\in S\setminus B$. Without loss of generality we assume that $B=\{\{1,2,3\}, \{1,4,5\}, \{1,6,7\}, \dots, \{1,n-1,n\}\}$ and $T_0=\{2,4,6\}$. The triples of ${\bar S}$ are parted into the following three sets of triples: $P(T_0)\cup \bigcup\limits_{T'\in B}P(T')$,  $\bigcup\limits_{T'\in S\setminus (T_0\cup B)}P(T')$ and $\bigcup\limits_{i=1}^{n}\{i,\bar{i},2n+1\}$.  
We consequently define the vector $v$ on these sets.

	Let us consider an auxiliary vector $w$ indexed by the triples from $ \{T_0\} \cup B$.
	For $i$, $1\leq i \leq \frac{n-1}{2}$, and even $\frac{n-1}{2}$, we define $w_{\{2,4,6\}}=-2$ and
	$$
		w_{\{1,2i,2i+1\}}=\begin{cases}			
		1, &1\leq i\leq \frac{n-1}{4}+1 \\
		-1,&\text{otherwise.}
	\end{cases} 	
	$$    
	And for $i$, $1\leq i \leq \frac{n-1}{2}$, and odd $\frac{n-1}{2}$, we define $w_{\{2,4,6\}}=-2$ and
	$$
	w_{\{1,2i,2i+1\}}=\begin{cases}			
		1, &1\leq i\leq \frac{n-3}{4} \\
		-1, &\frac{n+1}{4}\leq i\leq \frac{n-3}{2} \\
		2,&i=\frac{n-1}{2}.
	\end{cases} 	
	$$

The choice of $w$ implies that for any $i\in\{1,\ldots,n\}$

 \begin{equation}\label{w_proj}
\sum\limits_{T'\in \{T_0\}\cup B, i\in T' } w(T')\in\{\pm 1,\pm 2\},
\end{equation}
%!

\begin{equation}\label{w_sum}
\sum_{T\in  \{T_0\} \cup B} w_T=0.
\end{equation}

%!B_0\rightarrow B	

%	
%	\begin{center}
%		
%		\begin{tikzpicture}[thick, scale=1, every node/.style={scale=0.9}]
%			%Construction for 1/2
%			
%			%disctance between vetices in one line
%			\def \l {1} 
%			
%			%distance between w and v111
%			\def \d {3.3}
%			
%			%BLOCK 1
%			%angle from Axe x for Block 1
%			\def \alfabf {30}
%			
%			
%			\node[draw, circle] (w1) at (0, -2*\l) {$1$};
%			
%			
%			\node[draw, circle] (w2) at (0, \l) {$2$};
%			\node[draw, circle] (w3) at (\l, \l) {$3$};
%			
%			\node[draw, circle] (w4) at (0+2.5\l, \l) {$4$};
%			\node[draw, circle] (w5) at (3.5\l, \l) {$5$};
%			
%			
%			
%			
%		
%			\draw [fill=black!100, color=black!100] ( 4.0\l,\l ) circle (0.01cm);
%			\draw [fill=black!100, color=black!100] ( 4.25\l,\l ) circle (0.01cm);
%			\draw [fill=black!100, color=black!100] ( 4.5\l,\l ) circle (0.01cm);
%			
%			\node[draw, circle] (w6) at (5.0\l, \l) {$\frac{n+3}{2}$};
%			\node[draw, circle] (w7) at (6.0\l, \l) {$\frac{n+5}{2}$};
%			
%			\node[draw, circle] (w6) at (8.0\l, \l) {$\frac{n+7}{2}$};
%			\node[draw, circle] (w7) at (9.0\l, \l) {$\frac{n+9}{2}$};
%			
%			
%		
%			
%			\node[below,font=\large\bfseries] at (current bounding box.south) {Figure 2. Graph $L(k,m)$};
%			
%		\end{tikzpicture}
%	\end{center}
%	

	Let us now define the elements of $v$ on the triples  arising from $\{T_0\}\cup B$ in Assmuss-Mattson recursive approach. For $T\in \{T_0\}\cup B$,  put $v_{T'}=w_T$ for all $T'\in P(T)$.
	
	Define the numbers $\alpha_s=\sum\limits_{T'\in P(T_0)\cup \bigcup\limits_{T\in B}P(T), s\in T' }{v_{T'}}$ for $s\in \{1,\dots,n,\bar{1},\dots,\bar{n}\}$.
%!
 From the definition of $P(T)$ we see that for any $i\in T$  the points $i$ and $\bar{i}$ are in exactly two triples of $P(T)$.  Since $v(T')$ is the same for $T'\in P(T)$, we see that for $i\in\{1,\ldots,n\}$ we have $$\alpha_i=\alpha_{\bar{i}}= \sum\limits_{T'\in P(T_0)\cup \bigcup\limits_{T\in B}P(T),i\in T'}{v_{T'}}.$$
	 
Because for any $T'\in S$, $i\in T'$, the value of  $w_{T'}$ is doubled in $\alpha_i$ and doubled in $\alpha_{\bar{i}}$, using (\ref{w_proj}) we obtain that:
	\begin{equation}\label{E:e1}
		\alpha_i \in \{\pm 2,\pm 4\}.
	\end{equation}
The equality (\ref{w_sum}) gives the following (note that so far, some values of $v$ are still zeros):
	\begin{equation}\label{E:e2}
		\sum_{i \in \{1,\ldots,n\}}\alpha_i =0.
	\end{equation}	 
By Lemma \ref{L:4covering} there is a function $h$ defined on $S$ such that for any $\{i,j,k\} \in S$, $h(\{i,j,k\})\in \{i,j,k\}$, and for any $i\in \{1,\ldots,n\}$, $|\{T\in S: h(T)=i\}|\geq 4$. Take any point $j\in \{2,\ldots,n\}$.  
Clearly, $ \{T_0\}\cup B$ covers every point at most twice. Consequently,  $$|M_j|\geq 2,$$
where $M_j=\{T\in S\setminus(\{T_0\}\cup B): h(T)=j\}$.
  The next step is to apply Proposition \ref{P:zero}  in order to define $v$ on $P(T)$ and $P'(T)$ for $T\in M_j$. If the size of $M_j$ is even then we divide $M_j$ into two sets $M^1_j$ and $M^2_j$ of equal cardinality. After that we define  values of $v$ as follows. For a triple $T$ in $M_j^1$, judging by the value of $\tau$, we set the values $v_{T'}$ to be +1 and -1 for all triples in $T'\in P(T)$ as the values of vector $g$ in Proposition \ref{P:zero} with $a_1=j$. For a triple $T$ in $M_j^2$ we set the values $v_{T'}$ to be +1 and -1 for all triples in $T'\in P(T)$ as the values of vector $-g$ in Proposition \ref{P:zero} with $a_1=j$.

Due to definition of $v_{T'}$ from $g(T')$ we have:

$$\sum_{ T'\in P(T): j\in T',T \in M_j} v_{T'}=\sum_{T'\in P(T): j\in T', T \in M_j^1}g_{T'}-\sum_{T'\in P(T): j\in T, T \in M_j^2}g_{T'}=$$ 
$$\\ \sum_{T \in M_j^1}\sum_{T'\in P(T): j\in T'}g(T')-\sum_{T \in M_j^2,}\sum_{T'\in P(T): j\in T'}g(T');$$

Taking into account equality (\ref{eqgvec11}) and because $|M_j^1|=|M^2_j|$, we see that $$\sum_{T \in M_j^1}\sum_{T'\in P(T): j\in T'}g(T')-\sum_{T \in M_j^2,}\sum_{T'\in P(T): j\in T'}g(T')=\sum_{T \in M_j^1}2-\sum_{T \in M_j^2}2=0.$$

The same holds for the point $\bar{j}$ as we use (\ref{eqgvec12}) to obtain the following:

$$\sum_{ T'\in P(T): \bar{j}\in T',T \in M_j} v_{T'}=\\ \sum_{T \in M_j^1}\sum_{T'\in P(T): \bar{j}\in T'}g(T')-\sum_{T \in M_j^2,}\sum_{T'\in P(T): \bar{j}\in T'}g(T')=$$
$$\sum_{T \in M_j^1}-2+\sum_{T \in M_j^2}2=0.$$

Thus, we have: 
\begin{equation}\label{eq11}
\sum_{T': T'\in P(T), T \in M_j, j\in T'} v_{T'}=\sum_{T': T'\in P(T), T \in M_j, \bar{j}\in T'} v_{T'}=0. 
\end{equation}

Moreover, according to (\ref{eqgvec2}) we see that for the "projection" of any point different from $j$  is zero. For any $s\in T\cup \{\bar{i}:i \in T\}\setminus \{j,\bar{j}\}$ we have that:
\begin{equation}\label{eq12}
\sum_{T': T'\in P(T), T \in M_j,  \bar{s}\in T'} v_{T'}=\sum_{T': T'\in P(T), T \in M_j,  \bar{s}\in T'} g_{T'}=0.
\end{equation}

In the case of odd size of the set $M_j$, we divide $M_j$ into three non-intersecting subsets $M_j^1$, $M_j^2$ and $M_j^3$ respectively of sizes $\frac{|M_j|+1}{2}$, $\frac{|M_j|-3}{2}$ and $1$. 
We repeat the procedure (as for the case when $M_j$ was of even size) for the first two sets with vectors $g$ and $-g$ correspondingly. For the last one-element set we do the same but with the vector $-2g$. Similarly to the case of even size of $M_j$, we have equalities (\ref{eq11}) and (\ref{eq12}). We repeat the arguments for sets $M_t$ for all remaining $t\in \{1,\ldots,n\}\setminus \{1,j\}$.  

So far we have defined the elements of v indexed by all triples in ${\bar S}\setminus \{\{i,\bar{i},2n+1\}, i=1,\ldots,n\}$ and they are nonzeros. 
Moreover, from (\ref{eq11}), (\ref{eq12}) and (\ref{E:e1}) for any $i\in\{1,\ldots,n\}$

\begin{equation}\label{z0}\sum_{T'\in {\bar S}\setminus \{\{i,\bar{i},2n+1\}, i=1,\ldots,n\},i\in {\bar S}} v_{T'}= 
\sum_{T'\in {\bar S}\setminus \{\{i,\bar{i},2n+1\}, i=1,\ldots,n\},\bar{i}\in {\bar S}} v_{T'}=\alpha_i\in \{\pm 2,\pm 4\}.\end{equation}

%$$f(T)= 0 \text{ for } T\in \bigcup_{l=1,2,\dots, n}{\{l,\bar{l},2n+1\}},$$
%and for all $i \in [n]$ we have 

%$$	\alpha_i \in \{\pm 2,\pm 4\}, $$

%$$	\sum_{i \in [n]}\alpha_i =0, $$
%where  $\alpha_i=\sum_{T\in {\bar S}, i \in T}{f(T)} =\sum_{T\in {\bar S}, \bar{i} \in T}{f(T)}=\alpha_{\bar{i}}$.

%Since all our triples did not contain the element $2n+1$, we also have that $\sum_{T\in {\bar S}, 2n+1 \in T}{f(T)}=0$.
 
The last step that finishes the proof is to define $v$ on the set of triples $$\bigcup\limits_{i=1,2,\dots, n}{\{i,\bar{i},2n+1\}}$$
in the following way:
\begin{equation}\label{z1}v_{\{i,\bar{i},2n+1\}}=-\alpha_i \text{ for } i\in \{1,\ldots,n\}.\end{equation}
From (\ref{E:e2}) we see that \begin{equation}\label{z2}\sum\limits_{\{i,\bar{i},2n+1\},l=1,\ldots,n}v_{\{i,\bar{i},2n+1\}}=\sum\limits_{i=1}^{n}\alpha_i=0.\end{equation}

Summing up the above we have the following.

1. All values of $v$  are nonzeros with absolute values not greater than $4$.	

2. From (\ref{z0}) and (\ref{z1}), for any $i\in\{1,\ldots,n\}$  $\sum\limits_{T'\in {\bar S},i\in T'} v_{T'}=\sum\limits_{T'\in {\bar S},{\bar i}\in T'} v_{T'}=0$.
       
3. From (\ref{z2}) the equality $\sum\limits_{T'\in {\bar S},2n+1\in T'} v_{T'}=0$ holds.

In other words, $v$ is a $5$-flow for ${\bar S}$.
\end{proof}

\section{Completely regular codes in the block graphs of STSs}
Let $S$ be a Steiner triple system in the classical sense of order $n$.
The following are examples of completely regular codes in the block graph of $S$. 

 {\it Covering radius $\rho=1$ and eigenvalue }$\theta_1(\Gamma)$:
 
Construction 1. $\{B\in S: i \in B\}$, where $i$ is any fixed point  $\{1,\ldots,n\}$.

Construction 2. Any Steiner subsystem of $S$ having order $\frac{n-1}{2}$. 

Construction 3. $\{B\in S: i \in B\} \cup S'$, where $S'$ is any Steiner subsystem of $S$ having order $\frac{n-1}{2}$, such that $i$ is a point of $S$ but not a point of $S'$.

{\it Covering radius $\rho=1$ and eigenvalue} $\theta_2(\Gamma)$:

Construction 4 (see Proposition \ref{CRCeg2}). Any $1$-subdesign of $S$.
 
{\it Covering radius} $\rho=2$:

Construction 5. Any Steiner subsystem of $S$ of order less then $\frac{n-1}{2}$. 

{\bf Remark 2.} Actually, Construction 5 lists all completely regular codes in the block graphs with $\rho=2$. This can be shown, for example, using the technique from \cite[Theorem 4 and Lemma 2]{MogVor} that
utilizes the fact that all such codes naturally arise from subsets of the vertices of the clique graph of the block graph. This is beyond the scope of the current study, so we skip the details here.

The block graph of the projective (Hamming) Steiner triple system of order $2^r-1$ is isomorphic to the Grassmann graph $J_2(r,2)$. 
The completely regular codes with $\rho=1$ and the first eigenvalue in these graphs are known as Cameron-Liebler line classes.
These objects were characterized in \cite{Drudge} as follows: these  are Constructions 1-3 or their opposite codes.  
Judging by this fact for the most "symmetric"\,Steiner triple system, we propose the following:

{\bf Problem 2.} Find any other completely regular codes with $\rho=1$ and the first eigenvalue in the block graphs of Steiner triple systems of order $n, n\geq 13$ or prove that no
such codes exist.

All Steiner triple systems of orders $13$ and $15$ are enumerated and there are $2$ and $80$ isomorphism classes of such Steiner triple systems respectively \cite{CCW}.

\begin{theorem}
Let $S$ be a Steiner triple system of order $13$ or $15$.
Then all completely regular codes with $\rho=1$ in $\Gamma_S$ and eigenvalue $\theta_1(\Gamma_S)$ are codes from Constructions 1-3.
\end{theorem}
\begin{proof}
For a given Steiner triple system of order $n$, the number of codes from Construction 2 equals the number of Steiner subsystems of order $\frac{n-1}{2}$.
Using a well-known result of \cite{Douen} the number of such subsystems equals the $2^{n-r}-1$, where $r$ is the binary rank of Steiner triple system. We recall that the rank is the dimension of the subspace, spanned by the characteristic vectors of triples in the point set. We  see that there are exactly $(2^{n-r}-1)$ and $(2^{n-r}-1)\frac{n+1}{2}$ codes given by Constructions 2 and 3 respectively. 

We conclude that for a given Steiner triple system, the number of codes from Constructions 1, 2, 3 are as follows:
\begin{equation}\label{numconstr}n+ (2^{n-r}-1)(\frac{n+3}{2}).\end{equation} 

We use integer based computer search for completely regular codes with $\rho=1$, described in \cite{Mog}.
Given the intersection array, a computer linear programming solver outputs the number of completely regular codes having this intersection array.
For all considered Steiner triple systems of orders $n=13$ and $15$ of any given rank $r$, the solver output the number of completely regular codes with eigenvalue $\theta_1(\Gamma_S)$ being equal to (\ref{numconstr}), thus we have the required.
 \end{proof}

{\bf Acknowledgement.} The authors would like to thank the referee for careful reading and useful remarks that improved the representation of the paper.

\end{document}